\documentclass[a4paper]{article}
\usepackage{amsmath,amsfonts,amssymb,amsthm}
\usepackage{graphicx}
\usepackage{tikz}
\usetikzlibrary{cd,positioning}
\usepackage{eucal}
\usepackage{makeidx}
\usepackage{stackrel}
\usepackage{thmtools}
\usepackage[colorlinks=true, allcolors=black, bookmarks=true,hypertexnames=false]{hyperref}
\usepackage{cleveref}

\theoremstyle{definition}
\newtheorem{definition}{Definition}[section]
\newtheorem{example}[definition]{Example}
\newtheorem{remark}[definition]{Remark}
\newtheorem{assumption}[definition]{Assumption}
\theoremstyle{plain}
\newtheorem{proposition}[definition]{Proposition}
\newtheorem{lemma}[definition]{Lemma}
\newtheorem{theorem}[definition]{Theorem}
\newtheorem{corollary}[definition]{Corollary}

\crefname{section}{Section}{Sections}
\crefname{definition}{Definition}{Definitions}
\crefname{proposition}{Proposition}{Propositions}
\crefname{lemma}{Lemma}{Lemmas}
\crefname{theorem}{Theorem}{Theorems}
\crefname{corollary}{Corollary}{Corollaries}
\crefname{remark}{Remark}{Remarks}
\crefname{example}{Example}{Examples}

\newcommand\Hom{\mathop{\mathrm{Hom}}\nolimits}
\newcommand\stHom{\mathop{\underline{\mathrm{Hom}}}\nolimits}
\newcommand\prHom{\mathop{\mathrm{Hom}^{\mathrm{pr}}}\nolimits}
\newcommand\End{\mathop{\mathrm{End}}\nolimits}
\newcommand\op{^\mathrm{op}}
\newcommand\Z{\mathbb{Z}}
\newcommand\derb{\mathop{D^b}}
\newcommand\modcat{\mathop{\mathrm{mod}}\nolimits}
\newcommand\stmodcat{\mathop{\underline{\mathrm{mod}}}\nolimits}
\newcommand\homob{\mathop{K^b}}
\newcommand\projcat{\mathop{\mathrm{proj}}}
\newcommand\addcat{\mathop{\mathrm{add}}}
\newcommand{\syzygy}[1][]{%
  \if\relax\detokenize{#1}\relax
    \mathop{\Omega}
  \else
    \mathop{\Omega^{\scriptstyle #1}}
  \fi
}

\title{%
Two-sided tilting complexes over symmetric algebras sending simple modules to shifts of modules  
}
\author{Shuji Fujino \and Yuta Kozakai}
\date{}
\begin{document}
\maketitle
\begin{abstract}
  For a tilting complex over a finite-dimensional symmetric algebra over an algebraically closed field, whose distinct graded components have no indecomposable projective direct summands in common,
  we explicitly (re)construct a two-sided tilting complex that corresponds to the tilting complex, starting from a bimodule inducing a stable equivalence of Morita type.
\end{abstract}
\section{Introduction}
Let $k$ be an algebraically closed field and $A$ and $B$ finite dimensional  symmetric $k$-algebras. One of the central topics in the representation theory of finite dimensional algebras is to determine whether two algebras $A$ and $B$ are derived equivalent or not. In \cite{Rickard19896}, Rickard showed that the algebras $A$ and $B$ are derived equivalent if and only if there exists a tilting complex $T$ of $A$-modules whose endomorphism algebra is isomorphic to the algebra $B$.

On the other hand, we say that a complex $C$ of $(B,A)$-bimodules is a two-sided tilting complex if the functor ${-} \otimes_{B}^{\mathbb L} C: \derb{(B)}\to\derb{(A)}$ induces a triangulated equivalence.
In \cite{Rickard1991, Keller1993},  Rickard and Keller showed that for the tilting complex $T$,  there theoretically exists a two-sided tilting complex $C$ of $(B, A)$-bimodules whose restriction to $A$ is isomorphic to the tilting complex $T$ in the derived category of $A$-modules.  Such equivalences sometimes make it easier to track the correspondences of complexes in derived categories.
This can be one of the motivations to construct two-sided tilting complexes from one-sided tilting complexes.

Stable equivalences of Morita type were introduced by Brou\'e in \cite{Broue1994}. They are defined by the existence of certain bimodules whose tensor functors induce equivalences between stable categories.
Rickard showed in \cite{Rickard1991} that derived equivalences induce stable equivalences of Morita type by taking minimal projective resolutions of two-sided tilting complexes. Consequently, for symmetric algebras, stable equivalences of Morita type are weaker than derived equivalences.

A natural problem is how to explicitly construct a two-sided tilting complex realizing a given stable equivalence of Morita type for concrete symmetric algebras. Although one-sided tilting complexes are often constructed explicitly in practice, analogous explicit constructions for two-sided tilting complexes are sometimes much less understood. Developing such constructions enables us to give a concrete description of the induced derived equivalences and to investigate derived equivalences between algebras.

There have been some studies on constructing two-term two-sided tilting complexes from bimodules inducing stable equivalences of Morita type by taking projective covers of bimodules, although their primary goal is not necessarily to obtain explicit two-sided tilting complexes. For example, Rouquier constructed a two-term two-sided tilting complex for blocks of group algebras with cyclic defect groups in order to
give a character correspondence in \cite{Rouquier1995}.

Moreover, in \cite{Okuyama2000}, Okuyama showed that tensoring successively with the constructed two-sided tilting complexes corresponds to the derived equivalences obtained by Okuyama's method described in \cite{okuyama1998derived}.

There are studies aimed at constructing explicit two-sided tilting complexes for specific classes of algebras, such as Brauer tree algebras \cite{Kozakai_Kunugi_2018, Kozakai2019}, and generalized Brauer tree algebras \cite{Fujino_Kozakai_Takamura_2025}, which are determined by combinatorial graph structures. These studies respectively start from specific (not necessarily two-term) tilting complexes constructed by Rickard for Brauer tree algebras \cite{Rickard198911}, by Rickard--Schaps for Brauer tree algebras \cite{Rickard_Schaps_2002}, and by Membrillo-Hern\'{a}ndez for generalized Brauer tree algebras \cite{MembrilloHernandez1997}, together with bimodules inducing stable equivalences of Morita type.

Motivated by these results, we aim to give an explicit description of two-sided tilting complexes corresponding to given tilting complexes under suitable assumptions.
Our approach is to construct such complexes by adding appropriate projective bimodules to each graded component of a stalk complex inducing a stable equivalence of Morita type.

Then, in this paper, we start with two algebras
$A$ and $B$ that are derived equivalent via a tilting complex $T$ of $A$-modules. Rather than asking whether the two algebras are derived equivalent, we investigate how this derived equivalence can be realized by giving an explicit construction of a two-sided tilting complex whose unique graded component induces a stable equivalence of Morita type.
We show that if the tilting complex
$T$ has no indecomposable projective summands in common among its different degrees, then we can construct a two-sided tilting complex by deleting certain direct summands from a minimal projective resolution of a bimodule inducing a specific stable equivalence of Morita type (see \cref{s:main}).
It is required that the images of simple modules under tensoring with the bimodule coincide with those images from the derived equivalence induced by the tilting complex $T$.
Since each graded component of the two-sided tilting complex is projective on both sides, the tensor functor directly induces a derived equivalence.
We remark that any two-term tilting complex, Rickard tree-to-star tilting complex, Rickard--Schaps tree-to-star tilting complex, Membrillo-Hern\'{a}ndez tree-to-star tilting complex has no indecomposable projective summands in common among its different degrees.
We also prove that the restriction of the constructed two-sided tilting complex to $A$ is isomorphic to the tilting complex
$T$ with the $A$-action twisted by an automorphism of the algebra $A$ (see \cref{thm:dualtwosided}).

\section{Notation}
Throughout this paper, $\Gamma$ and $\Lambda$ mean finite dimensional indecomposable symmetric algebras over an algebraically closed field $k$. Let $\Gamma^{\op}$ denote the opposite algebra of $\Gamma$. All $\Gamma$-modules are finitely generated right $\Gamma$-modules unless otherwise stated. We identify $(\Lambda,\Gamma)$-bimodules with $\Lambda^{\op}\otimes_k \Gamma$-modules, which  restricts to $\Lambda^{\op}$-modules or $\Gamma$-modules.
Given a $\Gamma$-module $U$, we denote by $P(U)$ a projective cover of $U$, by $\syzygy[]U$ the kernel of a projective cover of $U$, and by $\syzygy[-1] U$ the cokernel of an injective hull of $U$. We define inductively $\syzygy[n+1] U=\syzygy (\syzygy[n] U)$ and $\syzygy[-n-1] U=\syzygy[-1](\syzygy[-n] U)$ for each integer $n\ge 1$ and a $\Gamma$-module $U$.
We denote by $U^{\ast}:=\Hom_{\Gamma}(U,\Gamma)$ the dual module of a $\Gamma$-module $U$. This is a $\Gamma^{\op}$-module defined by $(af)(u)=f(ua)$ for all $a\in \Gamma^{\op}$, $f\in U^{\ast}$ and $u\in U$. We note that $V^{\ast}\otimes_k U$ is a $(\Lambda,\Gamma)$-bimodule for a $\Gamma$-module $U$ and {a} $\Lambda$-module $V$.
For automorphisms $\alpha, \beta$ of the algebra $\Gamma$, we denote ${}_{\alpha}\Gamma_{\beta}$ by $\Gamma^{\op}\otimes_k \Gamma$-module defined by $x(a\otimes_k b)=\alpha(a)x\beta(b)$ for $x\in \Gamma$ and $a\otimes_k b \in \Gamma^{\op}\otimes_k\Gamma$. We note that ${({}_{\alpha}\Gamma_{\beta})}^{\ast}$ is isomorphic to ${}_{\beta}\Gamma_{\alpha}$.
Let $\{C^n\}_{n\in\mathbb Z}$ be a family of $\Gamma$-modules indexed by $\Z$.
Let $d^n$ be a $\Gamma$-homomorphism from $C^n$ to $C^{n+1}$, which satisfies $d^{n+1}\circ d^n=0$ for each integer $n$. Then $C=(C^n, d^n)_{n\in \Z}$ is a {complex} of $\Gamma$-modules.
We define $C^{\ast}$ to be $({(C^{-n})}^{\ast}, \Hom_{\Gamma}(d^{-n-1},\Gamma):{(C^{-n})}^{\ast}\to {(C^{-n-1})}^{\ast})_{n\in \mathbb Z}$. The dual complex of $C$ is a complex of $\Gamma^{\op}$-modules.
We denote by $C[n]$ the $n$-shifted complex of $C$ and by $H^n(C)$ the $n$th cohomology of $C$ for each $n$.

We denote the category of finitely generated right $\Gamma$-modules by $\modcat \Gamma$, the full subcategory of $\modcat \Gamma$ whose objects are finitely generated right projective $\Gamma$-modules by $\projcat \Gamma$,
the bounded homotopy category of $\projcat \Gamma$ by $\homob(\projcat \Gamma)$,
the bounded homotopy category of $\modcat \Gamma$ by $\homob(\modcat \Gamma)$,
the bounded derived category of $\modcat \Gamma$ by $\derb (\Gamma)$, and the stable {module} category of $\Gamma$ by $\stmodcat{\Gamma}$.
We denote by $\addcat(T)$ the smallest full subcategory of $\homob(\projcat \Gamma)$ containing a complex $T$ closed under isomorphisms, direct sums and direct summands.

Given a set $X$, we denote by $\delta_{ij}$ the Kronecker delta for $i, j \in X$ given by \[\delta_{ij}=\begin{cases}
    1 & \text{if $i=j$,}     \\
    0 & \text{if $i\neq j$}.
  \end{cases}\]

\section{Preliminaries}\label{s:equivalences}
\subsection{Tilting complexes}\label{s:tilting}
Tilting complexes are important to consider derived equivalences. We describe them in this subsection.
Let $k$ be an algebraically closed field and $\Gamma$ and $\Lambda$ two finite dimensional indecomposable symmetric $k$-algebras. We say that $\Gamma$ and $\Lambda$ are derived equivalent if $\derb(\Gamma)$ and $\derb(\Lambda)$ are equivalent as triangulated categories.
\begin{definition}
  We call a bounded complex $T$ of projective $\Gamma$-modules a \emph{tilting complex} if the following conditions are satisfied:
  \begin{itemize}
    \item $\Hom_{\homob (\projcat \Gamma)}(T, T[n])=0$ for any non-zero integer $n$.
    \item The subcategory $\addcat(T)$ generates $\homob (\projcat \Gamma)$ as a triangulated category.
  \end{itemize}
\end{definition}
The second condition means that $\Gamma$ is obtained by applying a finite sequence of operations for $T$, including taking direct sums, direct summands, mapping cones, and shifts. We recall the definition of two-sided tilting complexes.
\begin{definition}[{\cite[Definition 3.4]{Rickard1991}}]
  We call a bounded complex $C$ of $\Lambda^{\op}\otimes_k \Gamma$-modules a \emph{two-sided tilting complex} if
  ${-} \otimes_\Lambda^{\mathbb L} C$
  is a triangulated functor $\derb(\Lambda)\to \derb(\Gamma)$ which induces an equivalence.
\end{definition}
We note that ${-} \otimes_\Lambda^{\mathbb L} C$ is the left derived functor for ${-} \otimes_\Lambda C$.
This definition of $C$ is equivalent to saying that $C^{\ast}\otimes_{\Lambda}^{\mathbb L}C$ is isomorphic to $\Gamma$ in $\derb{(\Gamma^{\op}\otimes_k \Gamma)}$ and $C\otimes_{\Gamma}^{\mathbb L}C^{\ast}$ is isomorphic to $\Lambda$ in $\derb{(\Lambda^{\op}\otimes_k \Lambda)}$.
The following proposition holds.
\begin{proposition}[{\cite[Sec.~9.2.2]{Rickard1998}}]\label{prop:omittingL}
  Let $C$ be a two-sided tilting complex of $\Lambda^{\op}\otimes_k \Gamma$-modules  which are projective when seen as $\Lambda^{\op}$-modules and $\Gamma$-modules. Then ${-} \otimes_\Lambda^{\mathbb L} C$ is equivalent to ${-} \otimes_\Lambda C$ as a functor.
\end{proposition}

The following proposition establishes a relationship between tilting complexes, two-sided tilting complexes, and derived equivalences.

\begin{proposition}[{\cite[Theorem 6.4 and Corollary 8.3]{Rickard19896}, \cite[Theorem 3.3]{Rickard1991}, \cite[Theorem]{Keller1993}}]\label{prop:Rickard}
  The following conditions are equivalent:
  \begin{itemize}
    \item The $k$-algebra $\Gamma$ is derived equivalent to $\Lambda$.
    \item There exists a tilting complex $T$ of $\Gamma$-modules such that $\End_{\homob(\projcat{\Gamma})}(T)$ is isomorphic to $\Lambda$ as a $k$-algebra.
    \item There exists a two-sided tilting complex $C$ of $\Lambda^{\op}\otimes_k \Gamma$-modules.
  \end{itemize}
\end{proposition}
Let $T$ be a tilting complex of $\Gamma$-modules and set $\Lambda= \End_{\homob(\projcat{\Gamma})}(T)$. Then an equivalence from $\derb{(\Gamma)}$ to $\derb{(\Lambda)}$ is obtained by taking a $T$-resolution of an object in $\derb{(\Gamma)}$ and then applying the functor $\Hom_{\homob(\projcat{\Gamma})}(T, -)$. We say that this equivalence is induced by the tilting complex $T$.
\begin{proposition}[{\cite[Proposition 3.1]{Rickard1991}}]\label{prop:correspond}
  Let $T$ be a tilting complex of $\Gamma$-modules and $F:\derb{(\Gamma)}\to\derb{(\Lambda)}$ a derived equivalence induced by $T$. Then, there exists a two-sided tilting complex $C$ which restricts to $T$ in $\derb(\Gamma)$ and to $F(\Gamma)^{\ast}$ in $\derb(\Lambda^{\op})$.
\end{proposition}
We say that a two-sided tilting complex $C$ corresponds to $T$ if it satisfies the condition in \cref{prop:correspond}.

\begin{proposition}[{\cite[Lemma 1.3~(1)]{okuyama1998derived}}]\label{prop:imageofsimples}
  Let $T$ be a tilting complex of $\Gamma$-modules. Let $F:\derb{(\Gamma)}\to\derb{(\Lambda)}$ be a derived equivalence induced by $T$. We fix $X\in \derb{(\Gamma)}$ and $\ell_0 \in \mathbb Z $. If $\Hom_{\homob{(\modcat \Gamma)}}(T, X[\ell])=0$ for all $\ell \in \mathbb Z$ with $\ell \neq \ell_0$, then
  \[H^{\ell}(F(X))\cong \begin{cases}
      \Hom_{\homob{(\modcat \Gamma)}}(T, X[\ell]) & (\text{for $\ell = \ell_0$}),        \\
      0                                           & (\text{for all $\ell \neq \ell_0$}),
    \end{cases}\]
  as $\Lambda$-modules.
\end{proposition}
For $X, Y \in \derb{(\Gamma)}$, we write $\mathbb{R}\Hom_{\Gamma}(X,Y)$ for the derived Hom complex.
\begin{proposition}[{\cite[Lemma 3.7.10]{Zimmermann2014}, \cite[p.~399]{Weibel1994}}]\label{prop:imageofcomps}
  Let $X, Y\in \derb{(\Gamma)}$. Then for any integer $\ell$, we have
  \[H^{\ell}(\mathbb R\Hom_{\Gamma} (X,Y)) \cong \Hom_{\homob{(\modcat \Gamma)}}(X, Y[\ell]). \]

\end{proposition}

\subsection{Images of simple modules}
Throughout this subsection,
let $T=(T^{\ell}, d^{\ell})_{\ell \in \mathbb Z}=\bigoplus_{i=1}^n T_i$ be a basic tilting complex of $\Gamma$-modules, where each complex $T_i=(T_i^{\ell}, d_i^{\ell})_{\ell\in \mathbb Z}$ is indecomposable and not homotopy equivalent to $0$. Thus, $d_i^{\ell}$ is a radical homomorphism. We denote a triangulated  equivalence from $\derb(\Gamma)$ to $\derb(\Lambda)$ induced by the tilting complex $T$ by $F$.
We denote the pair-wise non-isomorphic simple $\Gamma$-modules by $S_1, \dots, S_n$.

\begin{proposition}\label{prop:precharacterization}
  The following are equivalent for each $i \in \{1, \dots, n\}$.
  \begin{enumerate}
    \item\label{eachtrivcap} $P(S_i)$ appears as a direct summand of a unique graded component of $T$.

    \item\label{eachmodshiftB} The complex $F(S_i)$ is isomorphic to a shift of a $\Lambda$-module.
  \end{enumerate}
  Moreover, if $P(S_i)$ is a summand of the unique graded component $T^{n_i}$, then $F(S_i)$ is isomorphic to $\Hom_{K^b(\modcat \Gamma)}(T,S_i[-n_i])[n_i]$ and vice versa.
\end{proposition}

\begin{proof}
  (\ref{eachtrivcap} $\Rightarrow$ \ref{eachmodshiftB}).
  We have $\Hom_\Gamma(T^{\ell}, S_i)=0$ for all $\ell\neq n_i$ by the assumption. Thus we have $\Hom_{\homob(\modcat \Gamma)}(T, S_i[-\ell])=0$ for all $\ell\neq n_i$.
  By \cref{prop:imageofsimples}, $F(S_i)\cong W_i[n_i]$, where we set a $\Lambda$-module $W_i=\Hom_{\homob(\modcat \Gamma)}(T, S_i[-n_i])$.

  (\ref{eachmodshiftB} $\Rightarrow$ \ref{eachtrivcap}).
  Since $F(S_i)\cong W_i[n_i]$ by the assumption, we have
  \[H^{\ell}(F(S_i)) \cong \begin{cases}
      W_i & (\text{for $\ell=-n_i$}),     \\
      0   & (\text{for $\ell\neq -n_i$}).
    \end{cases}\]
  By \cref{prop:imageofcomps},
  \[H^{\ell}(F(S_i))\cong H^{\ell}(\mathbb R\Hom_{\Gamma} (T,S_i)) \cong \Hom_{\homob{(\modcat \Gamma)}}(T, S_i[\ell]). \]
  Since $S_i$ is a simple module, we have \[\Hom_{\homob{(\modcat \Gamma)}}(T, S_i[\ell]) \cong \Hom_{\Gamma}(T^{-\ell}, S_i). \]
  Thus \[\Hom_\Gamma(T^{\ell}, S_i)\cong \begin{cases}
      W_i & (\text{for $\ell = n_i$}),    \\
      0   & (\text{for $\ell \neq n_i$}).
    \end{cases}\]
  Therefore,
  $P(S_i)$ is a summand of $T^{n_i}$ but not a summand of $T^{\ell}$ for $\ell \neq n_i$.
\end{proof}

\begin{proposition}\label{prop:characterization}
  The following are equivalent:
  \begin{enumerate}
    \item\label{trivcap} $\addcat{T^{\ell}}\cap \addcat{T^{\ell'}} = 0$ if $\ell \neq \ell'$.
    \item \label{modshiftB} For each $i$, there exist an indecomposable $\Lambda$-module $W_i$ and an integer $n_i$ such that $F(S_i) = W_i[n_i]$.
  \end{enumerate}
\end{proposition}

\begin{proof}
  By \cref{prop:precharacterization}, condition~\ref{modshiftB} is equivalent to the following: for each $i$, there exists an integer $n_i$ such that
  \[
    P(S_i) \in \addcat{T^{\ell}} \quad \text{if and only if} \quad \ell = n_i.
  \]
  Thus, if $P(S_i)\in \addcat{T^{\ell}}\cap \addcat{T^{\ell'}}$, then $\ell =\ell' = n_i$. By using this fact,  we can easily check that the two conditions are equivalent.
\end{proof}

\subsection{Stable equivalences of Morita type}\label{s:stable}
In this subsection, we recall the basic results on stable equivalences for symmetric algebras, which are weaker equivalences than derived equivalences.
Let $\Gamma$ and $\Lambda$ be two finite dimensional indecomposable symmetric $k$-algebras. For $\Gamma$-modules $U$ and $V$, we denote the $k$-linear space of all homomorphisms from $U$ to $V$ which factor through projective modules by $\prHom(U,V)$. The stable category of $\Gamma$-modules denoted by $\stmodcat{\Gamma}$ is defined as follows:
\begin{itemize}
  \item The objects are the same  as those of $\modcat \Gamma$.
  \item For $\Gamma$-modules $U$ and $V$, the set of morphisms from $U$ to $V$ is $\Hom(U,V)/\prHom(U,V)$. We denote this by $\stHom(U,V)$.
\end{itemize} In addition, the category $\stmodcat{\Gamma}$ is a triangulated category with the shift functor $\syzygy[-1]$.

\begin{definition}[{\cite[\S5]{Broue1994}}]
  We say that $\Gamma$ and $\Lambda$ are \emph{stably equivalent of Morita type} if there exist a $(\Lambda, \Gamma)$-bimodule $M$ and a $(\Gamma, \Lambda)$-bimodule $N$ satisfying the following conditions.
  \begin{itemize}
    \item The bimodules $M$ and $N$ are projective as left modules and right modules.
    \item $N\otimes_\Lambda M\cong \Gamma\oplus P$ as $(\Gamma,\Gamma)$-bimodules for some projective $(\Gamma,\Gamma)$-bimodule $P$.
    \item $M\otimes_\Gamma N\cong \Lambda\oplus Q$ as $(\Lambda,\Lambda)$-bimodules for some projective $(\Lambda,\Lambda)$-bimodule $Q$.
  \end{itemize}
  Then, we say that $M$ \emph{induces a stable equivalence of Morita type}.
\end{definition}

We remark that the above $(\Lambda, \Gamma)$-bimodule $M$ induces a functor ${-}\otimes_\Lambda M: \modcat{\Lambda} \to \modcat{\Gamma}$, which induces a stable equivalence $\stmodcat{\Lambda} \to \stmodcat{\Gamma}$.

Let $\mathcal S'$ denote a complete set of representatives of isomorphism classes of simple $\Lambda$-modules.

\begin{proposition}[{\cite[Lemma 2]{Rouquier1995}}]\label{prop:projcoverRouquier}
  Let $M$ be a $\Lambda^{\op}\otimes_k \Gamma$-module, which is projective as a $\Gamma$-module and as a
  $\Lambda^{\op}$-module. A projective cover of $M$ is isomorphic to \[\bigoplus_{V\in \mathcal S'} P(V)^{\ast} \otimes_k P(V\otimes_\Lambda M). \]
\end{proposition}

Let $M$ be an indecomposable $\Lambda^{\op}\otimes_k \Gamma$-module inducing a stable equivalence of Morita type between $\Lambda$ and $\Gamma$.
Let $P=(P^t, d_M^t)_{t\in \Z}$ be a projective resolution of the $\Lambda^{\op}\otimes_k \Gamma$-module $M$. Since $\syzygy[\ell] M$ gives a stable equivalence of Morita type for $\ell \ge 1$ (see \cite[Comparison 2.3.5]{Rouquier2001a} and \cite[Lemma 4.2]{Kozakai_Kunugi_2018}), the definition of a minimal projective resolution and \cref{prop:projcoverRouquier} give the following proposition.
\begin{proposition}\label{prop:projresol}
  The $(\Lambda,\Gamma)$-bimodule $P^{-t}$ is isomorphic to $M$ for $t=0$, to $0$ for all $t<0$, and to \[\bigoplus_{V\in \mathcal S'} P(V)^{\ast}\otimes_k P(V\otimes_{\Lambda} \syzygy[t-1] M)\] for all $t>0$.
\end{proposition}
The following proposition is useful to consider a minimal projective resolution of the module $M$.
\begin{proposition}[{\cite[Lemma 4.3]{Kozakai_Kunugi_2018}}]\label{prop:hellermove}
  We have $V \otimes_\Lambda \syzygy[\ell] M \cong \syzygy[\ell] (V \otimes_\Lambda M )$ for any simple $\Lambda$-module $V$ and $\ell\ge 0$.
\end{proposition}

The following proposition is proved similarly to those  propositions in \cite{Rouquier1998} and \cite{Kozakai_Kunugi_2018}. This is important to show that some complexes are two-sided tilting complexes.
\begin{proposition}[see {\cite[Sec.~10.3.4]{Rouquier1998} and \cite[Propositions 4.4 and 5.11]{Kozakai_Kunugi_2018}}]\label{prop:two-sided} Let $C=(C^{\ell}, d_C^{\ell})_{\ell \in \Z}$ be a bounded complex of $\Lambda^{\op}\otimes_k \Gamma$-modules such that $C^t=0$ for all $t>0$, $C^{0} = M$ and $C^t$ is projective for all $t<0$. If \[\Hom_{\derb(\Gamma)}(V\otimes_\Lambda C, W\otimes_\Lambda C[-n]) \cong \,\delta_{VW}\delta_{n0}k\] for $V,W\in \mathcal S'$ and $n\in \mathbb Z_{\ge 0}$, then $C$ is a two-sided tilting complex.
\end{proposition}

\section{Main results}\label{s:main}
In this section, we explicitly construct  a two-sided tilting complex corresponding to a tilting complex  without common projective summands. We take a bimodule inducing the stable equivalence of Morita type induced by the tilting complex. Then we take a minimal projective resolution of the bimodule and take a subcomplex of the resolution. We show that the subcomplex is a two-sided tilting complex.

Let $T=\bigoplus_{i=1}^n T_i$ be a basic tilting complex for an indecomposable symmetric algebra $A$, where each $T_i$ is indecomposable and not homotopy equivalent to $0$. We put $B=\End_{\homob(\projcat{A})}(T)$. We denote a triangulated  equivalence from $\derb(A)$ to $\derb(B)$ induced by the tilting complex $T$ by $F$. We denote a simple module corresponding to an indecomposable projective $B$-module $F(T_i)$ by $V_i$.
We make the following assumption on the derived equivalence induced by the tilting complex.
\begin{assumption}We assume that  image of each $S_i$ under the equivalence $F$ is isomorphic to a positive shift of a $B$-module.
\end{assumption}
For $i\in \{1,\dots, n\}$,  we set $n_i \ge 0$ and a $B$-module $W_i$ satisfying \[F(S_i)\cong W_i[n_i]\] in $\derb (B)$. In this situation, we remark that all the negative degrees of tilting complex $T$ are zero modules by \cref{prop:characterization}. By \cite{okuyama1998derived,Rickard1991}, there exists an $A^{\op}\otimes_k B$-module $M$ inducing a stable equivalence of Morita type satisfying
\[S_i\otimes_A M \cong \syzygy[-n_i] W_i \oplus (\text{projective}).\]

By \cite[Proposition 2.4 and Theorem 2.1 (ii)]{Linckelmann1996}, we may assume that $M$ is an indecomposable $A^{\op}\otimes_k B$-module and we have an isomorphism
\[
  S_i \otimes_A M \cong \syzygy[-n_i] W_i.
\]

Let $P^{\bullet}(M) = (P^j(M), p^j)_{j \in \mathbb Z}$ be a minimal projective resolution of the $A^{\op}\otimes_k B$-module $M$. That is, \[P^{-j}(M) \cong \begin{cases}
    M                 & (\text{for $j=0$}),    \\
    P(\syzygy[j-1] M) & (\text{for $0 < j$} ), \\
    0                 & (\text{for $j < 0$}).
  \end{cases}\]
By \cref{prop:projresol,prop:hellermove}, it holds that $P^{-j}(M)= P(\syzygy[j-1]M)$ is isomorphic to \[\bigoplus_{i=1}^n P(S_i)^{\ast}\otimes_k P(\syzygy[j-1-n_i] W_i)\]
for $j>0$.
Thus, for a simple $A$-module $S_i$ and for $j>0$, the $B$-module $S_i\otimes_A P^{-j}(M)$ is isomorphic to $P(\syzygy[j-1-n_i] W_i)$ because \[S_i\otimes_A P(S_{i'})^{\ast}\cong \Hom_A (P(S_{i'}), S_i)\cong \delta_{i'i}k.\] Since $P^{-j}(M)$ and $M$ are projective $A^{\op}$-modules, each short exact sequence of $A^{\op}\otimes_k B$-modules
\[\begin{tikzcd}[]
    0\rar[]&\syzygy[j] M\rar[]&P(\syzygy[j-1] M )\rar[]&\syzygy[j-1]M\rar[]&0
  \end{tikzcd}\]
splits as a short exact sequence of $A^{\op}$-modules.
Therefore, the exact sequence of $B$-modules
\[\begin{tikzcd}[]
    S_i\otimes_A \syzygy[j] M\rar[]&S_i\otimes_A P(\syzygy[j-1] M )\rar[]&S_i\otimes_A \syzygy[j-1]M\rar[]&0
  \end{tikzcd}\]
is a split short exact sequence of $k$-modules.
Thus $S_i\otimes_A P^{\bullet}(M)$ is an exact sequence, and hence a projective resolution of $S_i \otimes_A M $.
Moreover, since each graded component in the complex is minimal, this is in fact a minimal projective resolution of $S_i \otimes_A M $.

We construct a subcomplex $(C^j, p'^{j})_{j \in \mathbb Z}$ of $P^{\bullet}(M)$ as follows:
\[C^{-j} = \begin{cases}
    M                                                                                        & (\text{for $j=0$}),     \\
    \bigoplus_{i\in \{1,\dots, n\}, n_i\ge j} P(S_i)^{\ast}\otimes_k P(\syzygy[j-1-n_i] W_i) & (\text{for $0 < j $} ), \\
    0                                                                                        & (\text{for $j < 0$}).
  \end{cases}\]
For convenience of notation, we put $P^{-j}=P^{-j}(M)$.
Let $\iota^{-j}\colon C^{-j} \hookrightarrow P^{-j}$ be the inclusion and
$\pi^{-j}\colon P^{-j} \twoheadrightarrow C^{-j}$ the projection associated with
the direct summand $C^{-j}$ of $P^{-j}$.
We define a chain complex $C = (C^{j}, p'^{j})_{j \in \mathbb{Z}}$ by setting
\[
  p'^{-j} = \pi^{-j+1} \circ p^{-j} \circ \iota^{-j}.
\]
The following lemma holds.
\begin{lemma}\label{two-sided-image}
  For each $i$, we have $S_i\otimes_A C\cong F(S_i)$ in $\derb{(B)}.$
\end{lemma}
\begin{proof}
  We note that
  \[S_i\otimes_A C^{-j} \cong \begin{cases}
      \syzygy[-n_i] W_i       & (\text{for $j=0$}),                 \\
      P(\syzygy[j-1-n_i] W_i) & (\text{for $0 < j \le  n_i$} ),     \\
      0                       & (\text{for $j > n_i$ and $j < 0$}).
    \end{cases}\]
  and we have a commutative diagram \[\begin{tikzcd}[column sep=huge]
      S_i\otimes_A P^{-j}\rar["S_i\otimes_A p^{-j}"] &S_i\otimes_A P^{-j+1} \dar["S_i\otimes_A \pi^{-j+1}"]\\
      S_i\otimes_A C^{-j}\rar["S_i\otimes_A p'^{-j}"]\uar["S_i\otimes_A \iota^{-j}"] &S_i\otimes_A C^{-j+1}
    \end{tikzcd}.\]
  Let $D^{-j}$ be a complement of $C^{-j}$ in $P^{-j}$. The functor $S_i \otimes_A {-}$ sends a split short exact sequence \[\begin{tikzcd}[sep=large]
      0\rar[]&C^{-j}\rar["\iota^{-j}"]&P^{-j}\rar[]&D^{-j}\rar[]&0
    \end{tikzcd}\] to a split short exact sequence \[\begin{tikzcd}[sep=large]
      0\rar[]&S_i\otimes_A C^{-j}\rar["S_i\otimes_A \iota^{-j}"]&S_i\otimes_A P^{-j}\rar[]&S_i\otimes_A D^{-j}\rar[]&0
    \end{tikzcd}\] and similarly sends \[\begin{tikzcd}[sep=large]
      0\rar[]&D^{-j+1}\rar[]&P^{-j+1}\rar["\pi^{-j+1}"]&C^{-j+1}\rar[]&0
    \end{tikzcd}\] to  \[\begin{tikzcd}[sep=large]
      0\rar[]& S_{i}\otimes_A D^{-j+1}\rar[]&S_i\otimes_A P^{-j+1}\rar["S_i\otimes_A \pi^{-j+1}"]&S_i\otimes_A C^{-j+1}\rar[]&0.
    \end{tikzcd}\]
  By the definition of $D^{-j}$, we have \[D^{-j} =
    \bigoplus_{i\in \{1,\dots, n\}, n_i< j} P(S_i)^{\ast}\otimes_k P(\syzygy[j-1-n_i] W_i)\]
  for $0< j$. Therefore, if $0< j \le n_i$, \[S_i\otimes_A D^{-j} \cong  0 \qquad  \text{and} \qquad S_i\otimes_A D^{-j+1}\cong 0. \] Hence, the split morphisms $S_i\otimes_A \iota^{-j}$ and $S_i\otimes_A \pi^{-j+1}$ are isomorphisms. By the commutative diagram above, the complex $S_i\otimes_A C $ is a stupid truncated minimal projective resolution of $S_i\otimes_A M$ at the $-n_i$th degree. Therefore, \[S_i\otimes_A C\cong \syzygy[n_i](S_i\otimes_A M)[n_i] \cong \syzygy[n_i](\syzygy[-n_i] W_i)[n_i] \cong W_i[n_i]\cong F(S_i)\] in $\derb(B)$.
\end{proof}
By \cref{two-sided-image}, for $\ell\le 0$, \begin{align*}
  \Hom_{\derb(B)}(S_i\otimes_A C, S_j\otimes_A C[\ell]) & \cong \Hom_{\derb(B)}(F(S_i), F(S_j)[\ell]) \\ & \cong \Hom_{\derb(A)}(S_i, S_j[\ell])\\ & \cong \delta_{ij}\delta_{\ell0}k.
\end{align*}
By \cref{prop:two-sided}, we have the following proposition.

\begin{theorem}\label{thm:Cistwo-sided}
  The complex $C$ of $A^{\op}\otimes_k B$-modules is a two-sided tilting complex.
\end{theorem}

Let $G={-}\otimes_A C$. By \cref{thm:Cistwo-sided}, the functor $G$ is a triangulated equivalence. Since $(F^{-1}\circ G )(S_i)\cong S_i$ by \cref{two-sided-image}, the functor $F^{-1}\circ G$ sends the projective module $P(S_i)$ to itself. In particular, $(F^{-1}\circ G)(A)\cong A$.
Thus we have \[A\otimes_A C\cong G(A)\cong F(A)\] in $\derb{(B)}$. Hence, we have the following proposition.
\begin{proposition}\label{prop:invtwosided}
  The complex $C$ is isomorphic to the tilting complex $F(A)$  in $\derb{(B)}$.
\end{proposition}
Let $Y$ be a two-sided tilting complex which restricts to $T^{\ast}$ as $A^{\op}$-modules and to $F(A)$ as $B$-modules. Such $Y$ exists by \cref{prop:correspond}. By \cite[Proposition 2.3.]{Rouquier_Zimmermann_2003}, there exists a $k$-algebra automorphism $\alpha$ of $A$ such that ${}_{\alpha}A_{1}\!\otimes_A C\cong Y$ in $\derb(A^{op})$. By taking dual complexes, we have $C^{\ast}\otimes_A {}_{1}A_{\alpha} \cong Y^{\ast}\cong T$ in $\derb(A)$.
Therefore, we have the following theorem.
\begin{theorem}\label{thm:dualtwosided}
  There exists a $k$-algebra automorphism $\alpha$ of the algebra $A$ such that the complex $C^{\ast}\otimes_{A}{}_{1}\!A_{\alpha}$ is isomorphic to the tilting complex $T$ in $\derb{(A)}$.
\end{theorem}

We remark that the assumption $n_i \ge 0$ for all $i\in \{1, \dots, n\}$ is not essential. If there are some $i$ such that $n_i<0$, then we denote the minimum of $n_i$ for $i\in \{1, \dots, n\}$ by $N$. Then the bimodule $\syzygy[-N] M $ induces a stable equivalence of Morita type by \cite[2.3.5. Comparison]{Rouquier2001a}. We can do the same argument for the complex of $A^{\op}\otimes_k B$-modules which of the $N$th degree is isomorphic to $\syzygy[-N] M$.

The following corollary is obtained by
combining \cref{thm:dualtwosided} with \cref{prop:characterization}.
\begin{corollary}
  For a tilting complex $T=(T^{\ell}, d^{\ell})_{\ell \in \mathbb Z}$ satisfying $\addcat{T^{\ell}} \cap \addcat{T^{\ell'}} = 0$ for all $\ell \neq \ell'$, we can construct a corresponding two-sided tilting complex by deleting some indecomposable summands of a minimal projective resolution of some bimodule inducing a stable equivalence of Morita type.
\end{corollary}
\begin{remark}
  For Brauer tree algebras, tilting complexes in \cite{Rickard198911} and  tilting complexes in \cite{Rouquier_Zimmermann_2003} satisfy the conditions in \cref{prop:characterization}.

  For generalized Brauer tree algebras, tilting complexes in \cite{MembrilloHernandez1997} satisfy the conditions in \cref{prop:characterization}.

  For arbitrary finite dimensional   algebras, two-term tilting complexes satisfy the conditions in \cref{prop:characterization} by \cite[Proposition 2.5 and 3.6]{Adachi_Iyama_Reiten_2014}.

  We may regard that we generalize the construction method of two-sided tilting complexes in \cite{Kozakai_Kunugi_2018} corresponding to Rickard tree-to-star  tilting complexes for Brauer tree algebras, \cite{Kozakai2019} corresponding to Rickard--Schaps tree-to-star tilting complexes for Brauer tree algebras and \cite{Fujino_Kozakai_Takamura_2025} for Membrillo-Hernandez tree-to-star tilting complexes for generalized Brauer tree algebras.
\end{remark}

\begin{example}
  Let $A$ be the Brauer graph algebra associated with the following Brauer graph, with counterclockwise cyclic order around each vertex and all multiplicities equal to $1$.
  \tikzset{block/.style={circle,draw,minimum width=7, line width=0,minimum height=0,inner sep=1.0}}
  \tikzset{block1/.style={circle, white, fill=black,minimum width=0, line width=0,minimum height=0,inner sep=1.0}}
  \tikzset{block2/.style={circle, white, fill=black,minimum width=.8, line width=0,minimum height=0,inner sep=1.0}}
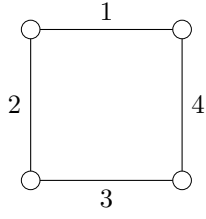
  \[\begin{tikzpicture}
      \node (A) at (1,1)[block, ]{};
      \node (B) at (-1,1)[block]{};
      \node (C) at (-1,-1)[block]{};
      \node (D) at (1,-1)[block]{};
      \path[]
      (A) edge node[above] {$1$} (B)
      (B) edge node[left] {$2$} (C)
      (C) edge node[below] {$3$} (D)
      (D) edge node[right] {$4$} (A);
    \end{tikzpicture}\]
  Silting mutation, introduced in \cite{Aihara_Iyama_2012}, provides a way to construct new silting objects from given ones in a triangulated category.
  We apply the right silting mutation to the tilting complex $A$,
  first with respect to $\{1\}$, and then with respect to $\{1,2,4\}$.
  As a result, we have the tilting complex $T=\bigoplus_{i=1}^4 T_i $ as follows:
  \[
    \begin{tikzcd}[ampersand replacement=\&, row sep=small, column sep=small ]
      \phantom{0}\&[-6pt]\phantom{0}\&[-5pt]\phantom{0} \& 0\text{th} \& 1\text{st} \& 2\text{nd}\\[-5pt]
      \&\&T_1\rar[":",phantom]\&P_3\oplus P_3\rar[]\&P_2\oplus P_4 \rar[]\& P_1\\
      \phantom{0}\&\&T_2\rar[":",phantom]\&P_3\rar[]\&P_2\rar[]\& 0\\
      \phantom{0}\&\&T_3\rar[":",phantom]\&P_3\rar[]\&0\rar[]\&0\\
      \&\&T_4\rar[":",phantom]\&P_3\rar[]\&P_4\rar[]\&0
    \end{tikzcd}
  \]
  Generalized Kauer move, introduced in \cite{Soto2024}, provides a way to calculate the endomorphism algebra of a tilting complex obtained by a silting mutation of a regular module over a Brauer graph algebra.
  The endomorphism algebra $B=\End_{\homob({\projcat A})}(T)$ is the Brauer graph algebra associated with the following Brauer graph:
  \tikzset{block/.style={circle,draw,minimum width=7, line width=0,minimum height=0,inner sep=1.0}}
  \tikzset{block1/.style={circle, white, fill=black,minimum width=0, line width=0,minimum height=0,inner sep=1.0}}
  \tikzset{block2/.style={circle, white, fill=black,minimum width=.8, line width=0,minimum height=0,inner sep=1.0}}
  \[\begin{tikzpicture}
      \node (A) at (1,1)[block, ]{};
      \node (B) at (-1,1)[block]{};
      \node (C) at (-1,-1)[block]{};
      \node (D) at (1,-1)[block]{};
      \path[]
      (B) edge node[left] {$2$} (D)
      (C) edge node[below] {$3$} (D);
      \draw (C) .. controls (-2,3.5) .. (D) node[midway, above right] {$1$};
      \draw (A) .. controls (2,-2.5) .. (C) node[midway, above right] {$4$};
    \end{tikzpicture}\]
  Following \cref{prop:precharacterization}, the images of simple modules $S_i$ through the derived equivalence $F$ induced by the tilting complex $T$ are as follows:
  \[F(S_1)\cong W_1[2],\quad F(S_2)\cong W_2[1],\quad F(S_3)\cong W_3,\quad F(S_4)\cong W_4[1].\]
  Here we put indecomposable modules $W_1,\dots,W_4$ by writing the composition factors from top to socle vertically as follows:
  \[W_1= V_1,\quad W_2= \,\begin{matrix}
      V_2 \\V_3
    \end{matrix}\,,\quad W_3= \,
    \begin{matrix}
      \begin{matrix}
        V_1 \\V_4
      \end{matrix}\quad \begin{matrix}
                          V_1 \\V_2
                        \end{matrix} \\V_3
    \end{matrix}\,
    ,\quad W_4= \,\begin{matrix}
      V_4 \\V_3
    \end{matrix}\,,\]
  where $V_1,\dots,V_4$ denote simple $B$-modules.
  Let $M$ be an $A^{\op}\otimes_k B$-module and induce a stable equivalence of Morita type satisfying \[S_1\otimes_A M\cong \syzygy[-2]W_1,\quad S_2\otimes_A M\cong \syzygy[-1]W_2,\quad S_3\otimes_A M\cong W_3,\quad S_4\otimes_A M \cong \syzygy[-1]W_4.\]
  Let $P^{\bullet}(M)$ be a minimal projective resolution of $M$.
  \[
    \begin{tikzcd}[row sep=0cm]
      \phantom{\cdots\cdots\cdots\cdots} &-2\text{nd}&-1\text{st}& 0\text{th}\\[-0.1cm]
      \phantom{\cdots} & {P(S_1)^{\ast}\otimes_k P(\syzygy[-1] W_1)}& P(S_1)^{\ast}\otimes_k P(\syzygy[-2] W_1)& \phantom{0}\\[-0.1cm]
      &|[opacity=1]|\oplus & \oplus &\\[-0.1cm]
      \phantom{\cdots}& |[opacity=1]|{P(S_2)^{\ast}\otimes_k P( W_2)}& P(S_2)^{\ast}\otimes_k P(\syzygy[-1] W_2)\\[-0.1cm]
      |[opacity=1]|{\cdots}\rar[shorten=30pt,opacity=1, ] &|[opacity=1]|{\oplus}\rar[shorten=35pt] & \oplus\rar[shorten <= 40pt,] &M\\[-0.1cm]
      \phantom{\cdots}& |[opacity=1]|{P(S_3)^{\ast}\otimes_k P(\syzygy[] W_3)}& |[opacity=1]|{P(S_3)^{\ast}\otimes_k P( W_3)}& \phantom{0}\\[-0.1cm]
      &|[opacity=1]|{\oplus}  & |[opacity=1]|{\oplus} &\\[-0.1cm]
      \phantom{\cdots}&|[opacity=1]|{P(S_4)^{\ast}\otimes_k P( W_4)}& P(S_4)^{\ast}\otimes_k P(\syzygy[-1] W_4)& \phantom{0}
    \end{tikzcd}
  \]
  Following our construction method of our two-sided tilting complexes, we delete some direct summands. Then we have a two-sided tilting complex $C$ of $A^{\op}\otimes_k B$-modules.
  \[
    \begin{tikzcd}[row sep=0cm]
      \phantom{\cdots\cdots\cdots\cdots} &-2\text{nd}&-1\text{st}& 0\text{th}\\[-0.1cm]
      \phantom{\cdots} & {P(S_1)^{\ast}\otimes_k P(\syzygy[-1] W_1)}& P(S_1)^{\ast}\otimes_k P(\syzygy[-2] W_1)& \phantom{0}\\[-0.1cm]
      &|[opacity=0]|\oplus & \oplus &\\[-0.1cm]
      \phantom{\cdots}& |[opacity=0]|{P(S_2)^{\ast}\otimes_k P( W_2)}& P(S_2)^{\ast}\otimes_k P(\syzygy[-1] W_2)\\[-0.1cm]
      |[opacity=0]|{\cdots}\rar[shorten=30pt,opacity=0, ] &|[opacity=0]|{\oplus}\rar[shorten=35pt] & \oplus\rar[shorten <= 40pt,] &M\\[-0.1cm]
      \phantom{\cdots}& |[opacity=0]|{P(S_3)^{\ast}\otimes_k P(\syzygy[] W_3)}& |[opacity=0]|{P(S_3)^{\ast}\otimes_k P( W_3)}& \phantom{0}\\[-0.1cm]
      &|[opacity=0]|{\oplus}  & |[opacity=0]|{\oplus} &\\[-0.1cm]
      \phantom{\cdots}&|[opacity=0]|{P(S_4)^{\ast}\otimes_k P( W_4)}& P(S_4)^{\ast}\otimes_k P(\syzygy[-1] W_4)& \phantom{0}
    \end{tikzcd}
  \]
  The dual complex $C^{\ast}$ is also a two-sided tilting complex of $B^{\op}\otimes_k A$-modules, which restricts to the tilting complex $T\otimes_{A}{}_{\alpha}\!A_{1}$ of $A$-modules for a $k$-algebra automorphism $\alpha$ of $A$.

  \[
    \begin{tikzcd}[
        row sep = 0
      ]
      & 0\text{th}
      & 1\text{st}
      & 2\text{nd}\\
      & \phantom{0}
      & P(\syzygy[-2] W_1)^{\ast}\otimes_k P(S_1)
      & P(\syzygy[-1] W_1)^{\ast}\otimes_k P(S_1)
      \\
      &
      & \oplus
      & |[opacity=0]| \oplus
      \\
      &
      & P(\syzygy[-1] W_2)^{*} \otimes_k P(S_2)
      & |[opacity=0]| {P(W_2)^{*} \otimes_k P(S_2)}
      \\
      & M^{\ast}
      \arrow[r, shorten >= 50pt]
      & \oplus
      \arrow[r, shorten = 35pt]
      & |[opacity=0]| {\oplus}
      \\
      &
      & |[opacity=0]| \oplus
      & |[opacity=0]| \oplus
      & \phantom{0} \\
      & |[opacity=0]| {P(W_3)^{*} \otimes_k P(S_3)}
      & |[opacity=0]| {P(\syzygy W_3)^{*} \otimes_k P(S_3)}
      \\
      & \phantom{0}
      & P(\syzygy[-1]W_4)^{*} \otimes_k P(S_4)
      & |[opacity=0]| {P(W_4)^{*} \otimes_k P(S_4)}
    \end{tikzcd}
  \]
\end{example}

\section*{Acknowledgements}
The authors would like to thank Professor Naoko Kunugi for her advice. This work was supported by JSPS KAKENHI Grant Number JP25K17238.

\vspace{\baselineskip}
\par\noindent
Shuji Fujino~~1124702@ed.tus.ac.jp\par\noindent
Department of Mathematics, Graduate School of Science, Tokyo University of Science, 1-3 Kagurazaka, Shinjuku-ku, Tokyo, 162-8601, Japan
\vspace{\baselineskip}
\par\noindent
Yuta Kozakai~~kozakai@rs.tus.ac.jp\par\noindent
Tokyo University of Science, 1-3 Kagurazaka, Shinjuku-ku, Tokyo 162-8601, Japan
\vspace{\baselineskip}

\end{document}